\documentclass[12pt,leqno,draft]{article} 
\usepackage{amsmath,amssymb,amsthm,amsfonts,bm}
\usepackage{enumerate,color}
\makeatletter
\def\@cite#1#2{[{{\bfseries #1}\if@tempswa , #2\fi}]}
\renewcommand{\section}{%
\@startsection{section}{1}{\z@}
{0.5truecm plus -1ex minus -.2ex}%
{1.0ex plus .2ex}{\bfseries\large}}
\def\@seccntformat#1{\csname the#1\endcsname.\ }
\makeatother

\numberwithin{equation}{section} 
\newtheorem{thm}{Theorem}[section]

\newtheorem{lem}[thm]{Lemma}

\theoremstyle{definition}
\newtheorem{df}{Definition}[section]
\newtheorem{remark}{Remark}[section]
\newtheorem*{ex1}{Example 2.1 
(The porous media equation, the fast diffusion equation)}
\newtheorem*{ex2}{Example 2.2 (The Stefan problem)}
\newtheorem*{prth1.1a}{Proof of Theorem 1.1 (Existence and uniqueness)}
\newtheorem*{prth1.1b}{Proof of Theorem 1.1 (Error estimate)}
\newtheorem*{prth1.2}{Proof of Theorem 1.2}
\newtheorem*{prth1.3}{Proof of Theorem 1.3}

\newcommand{\ep}{\varepsilon}
\newcommand{\pa}{\partial}
\newcommand{\RN}{\mathbb{R}^N}

\begin{document}
\footnote[0]
    {2010{\it Mathematics Subject Classification}\/. 
    Primary: 35K59, 35K35; Secondary: 47H05.
    }
\begin{center}
    \Large{{\bf Nonlinear diffusion equations as asymptotic limits of \\
Cahn--Hilliard systems on unbounded domains \\
via Cauchy's criterion
           }}
\end{center}
\vspace{5pt}
\begin{center}
    Takeshi Fukao%
    \footnote{Partially supported by Grant-in-Aid for
    Scientific Research (C), No.\,17K05321, JSPS.}\\
    \vspace{2pt}
    Department of Mathematics, 
    Kyoto University of Education\\ 
    1, Fujinomori, Fukakusa, Fushimi-ku, Kyoto 612-8522, Japan\\
    {\tt fukao@kyokyo-u.ac.jp}\\
    \vspace{12pt}
    Shunsuke Kurima\\
    \vspace{2pt}
    Department of Mathematics, 
    Tokyo University of Science\\
    1-3, Kagurazaka, Shinjuku-ku, Tokyo 162-8601, Japan\\
    {\tt shunsuke.kurima@gmail.com}\\
    \vspace{12pt}
    Tomomi Yokota%
   \footnote{Corresponding author}
   \footnote{Partially supported by Grant-in-Aid for
    Scientific Research (C), No.\,16K05182, JSPS.}\\
    \vspace{2pt}
    Department of Mathematics, 
    Tokyo University of Science\\
    1-3, Kagurazaka, Shinjuku-ku, Tokyo 162-8601, Japan\\
    {\tt yokota@rs.kagu.tus.ac.jp}\\
    \vspace{2pt}
\end{center}
\begin{center}    
    \small\today
\end{center}

\vspace{2pt}
\newenvironment{summary}
{\vspace{.5\baselineskip}\begin{list}{}{%
     \setlength{\baselineskip}{0.85\baselineskip}
     \setlength{\topsep}{0pt}
     \setlength{\leftmargin}{12mm}
     \setlength{\rightmargin}{12mm}
     \setlength{\listparindent}{0mm}
     \setlength{\itemindent}{\listparindent}
     \setlength{\parsep}{0pt}
     \item\relax}}{\end{list}\vspace{.5\baselineskip}}
\begin{summary}
{\footnotesize {\bf Abstract.}
    This paper develops an abstract theory for subdifferential operators 
    to give existence and uniqueness of solutions to    
    the initial-boundary problem \ref{P} for 
    the nonlinear diffusion equation in an {\it unbounded} domain 
    $\Omega\subset\RN$ ($N\in{\mathbb N}$), 
    written as
    \[
        \frac{\partial u}{\partial t} + (-\Delta+1)\beta(u) 
        = g \quad \mbox{in}\ \Omega\times(0, T), 
    \]
    %
    which represents the porous media, the fast diffusion  equations, etc., 
    where 
    $\beta$ is a single-valued maximal monotone function on 
    $\mathbb{R}$, and $T>0$.   
    In \cite{KY1} and \cite{KY2} existence and uniqueness for \ref{P} 
    were directly proved 
    under a growth condition for $\beta$ 
    even though the Stefan problem was excluded from examples of \ref{P}. 
    This paper completely removes the growth condition for $\beta$ 
    by confirming Cauchy's criterion for solutions of 
    the following approximate problem \ref{Pep} 
    with approximate parameter $\ep>0$: 
    \[
    \frac{\partial u_{\ep}}{\partial t} 
        + (-\Delta+1)(\ep(-\Delta+1)u_{\ep} + \beta(u_{\ep}) + \pi_{\ep}(u_{\ep}))
        = g \quad \mbox{in}\ \Omega\times(0, T), 
    \]
    which is called the Cahn--Hilliard system,  
    even if $\Omega \subset \RN$ ($N \in \mathbb{N}$) 
    is an {\it unbounded} domain. 
    Moreover, it can be seen that 
    the Stefan problem excluded from \cite{KY1} and \cite{KY2} is 
    covered in the framework of this paper.}
\end{summary}
\vspace{10pt}

\newpage

\section{Introduction and results} \label{Sec1}

Nonlinear diffusion equations have been studied since a long time ago. 
In particular, the problems on bounded domains and $\RN$ 
have been often considered for the equations. 
This paper will focus on the case of unbounded domains. 

In the case that $\Omega$ is a bounded domain 
in $\RN$, 
the nonlinear diffusion equation
\begin{align}\label{E}\tag*{(E)}
\frac{\partial u}{\partial t} -\Delta\beta(u) = g \quad \mbox{in}\ 
\Omega\times(0, T)
\end{align}
is studied by many mathematicians, 
where $\beta : \mathbb{R}\to\mathbb{R}$ is a maximal monotone function 
and $T>0$. 
Recently, in \cite{CF-2015} and \cite{CF-2016}  Colli and Fukao 
considered the Cahn--Hilliard type of approximate equation    
\begin{align}\label{Eep}\tag*{(E)$_{\ep}$}
\frac{\partial u_{\ep}}{\partial t} 
-\Delta(-\ep\Delta u_{\ep} + \beta(u_{\ep}) + \pi_{\ep}(u_{\ep})) = g 
\quad \mbox{in}\ 
\Omega\times(0, T), 
\end{align}
where $\pi_{\ep}$ is an anti-monotone function which goes to $0$ 
in some sense as $\ep \searrow 0$, 
and used one more approximation 
\begin{align}\label{Eeplamb}\tag*{(E)$_{\ep, \lambda}$}
\frac{\partial u_{\ep, \lambda}}{\partial t} 
-\Delta\Bigl(\lambda\frac{\partial u_{\ep, \lambda}}{\partial t}
           -\ep\Delta u_{\ep, \lambda} + \beta_{\lambda}(u_{\ep, \lambda}) 
           + \pi_{\ep}(u_{\ep, \lambda})\Bigr) 
= g \quad \mbox{in}\ \Omega\times(0, T), 
\end{align}
where $\beta_{\lambda}$ ($\lambda > 0$) 
is the Yosida approximation of $\beta$. 
They first in \cite{CF-2015} proved existence of solutions to \ref{Eeplamb} 
by the compactness method 
for doubly nonlinear evolution inclusions 
(see e.g., Colli and Visintin \cite{CV-1990}): 
    \[
    Au'(t) + \partial\psi(u(t)) \ni k(t)
    \]
with some bounded monotone operator $A$ and 
subdifferential operator $\partial\psi$ 
of a proper lower semicontinuous convex function $\psi$. 
They next in \cite{CF-2016} obtained existence of solutions to \ref{Eep} 
and \ref{E} by passing to the limit in \ref{Eeplamb} as $\lambda \searrow 0$ 
and in \ref{Eep} as $\ep \searrow 0$ individually. 
Although it is known that existence of solutions to \ref{E} 
can be directly proved 
under a growth condition for $\beta$ (see e.g., \cite[p.\ 205]{Barbu2}), 
in \cite[Section 6]{CF-2016} they used the above approach 
whose idea is based on the idea in Fukao \cite{F-2015} 
to obtain existence and estimates for \ref{E} 
without the growth condition for $\beta$ 
(the proof of existence of solutions to \ref{Eeplamb} does not need this condition),  
see also \cite{F-2015, F-2017} in the case of dynamic boundary conditions. 
A class of doubly nonlinear degenerate parabolic equations generalizing 
\ref{E} on bounded domains was studied by using maximal monotone operators 
in 
Damlamian \cite{D-1977}, 
Kenmochi \cite{Kenmochi-1990}, Kubo--Lu \cite{KL-2005} 
and so on; see also  
Droniou--Eymard--Talbot \cite{DET-2016}. 
Another approach to nonlinear diffusion equations via cross-diffusion systems 
was recently built by Murakawa \cite{Murakawa-2007, Murakawa-2017}, 
whose approach is versatile and easy-to-implement.
In comparison with the Cahn--Hilliard approximation as in 
\cite{CF-2016, F-2017}, 
the methods by \cite{D-1977, DET-2016, 
Kenmochi-1990, KL-2005, Murakawa-2007, Murakawa-2017} 
require the growth condition for $\beta$. 

On the other hand, in the case that 
$\Omega$ is an {\it unbounded} domain in $\RN$, 
nonlinear diffusion equations are not so sufficiently studied 
from a viewpoint of the operator theory, 
whereas in the case that $\Omega = \RN$ the equations are studied 
by the method of real analysis (see e.g., \cite{Deg}). 
The case of unbounded domains would be important 
in both mathematics and physics. 
This paper is concerned the initial-boundary value problem 
for nonlinear diffusion equations
%
%
 \begin{equation*}\tag*{(P)}\label{P}
     \begin{cases}
         \dfrac{\partial u}{\partial t}+(-\Delta+1)\beta(u) = g  
         & \mbox{in}\ \Omega\times(0, T),
 \\[2mm]
         \partial_{\nu}\beta(u) = 0                                   
         & \mbox{on}\ \partial\Omega\times(0, T),
 \\[1mm]
        u(0) = u_0                                         
         & \mbox{in}\ \Omega
     \end{cases}
 \end{equation*} 
by passing to the limit in the following 
Cahn--Hilliard system as $\ep \searrow 0$:
%
%
%
 \begin{equation*}\tag*{(P)$_{\ep}$}\label{Pep}
     \begin{cases}
         \dfrac{\partial u_{\ep}}{\partial t} 
         + (-\Delta+1)\mu_{\ep} = 0 
         & \mbox{in}\ \Omega\times(0, T),
     \\[2mm] 
         \mu_{\ep}=\ep(-\Delta+1)u_{\ep} 
         + \beta(u_{\ep})+\pi_{\ep}(u_{\ep})-f 
         & \mbox{in}\ \Omega\times(0, T),      
     \\[1mm]
         \partial_{\nu}{\mu_{\ep}} = \partial_{\nu}{u_{\ep}} = 0 
         & \mbox{on}\ \partial\Omega\times(0, T),
     \\[1mm]
        u_{\ep}(0) = u_{0\ep} 
        & \mbox{in}\ \Omega,
     \end{cases}
 \end{equation*}
where $\Omega$ is an {\it unbounded} domain 
in $\RN$ ($N \in{\mathbb N}$) with smooth bounded boundary $\partial\Omega$ 
(e.g., $\Omega = \mathbb{R}^{N}\setminus \overline{B(0, R)}$,  
where $B(0, R)$ is the open ball with center $0$ and radius $R>0$) or $\Omega=\mathbb{R}^{N}$ 
or $\Omega=\mathbb{R}_{+}^{N}$, 
$T>0$, and $\pa_\nu$ denotes differentiation with respect to 
the outward normal of $\pa\Omega$,  
under the conditions (C1)-(C4) given later. 
In this context there are two recent works \cite{KY1} and \cite{KY2} 
which dealt with \ref{P} and \ref{Pep} 
on unbounded domains. 
In \cite{KY1} and \cite{KY2} 
existence and estimates for \ref{P} 
could be directly proved by regarding \ref{P} as 
nonlinear evolution equations of the form 
 $$
 u'(t) + \partial\phi(u(t)) = \ell(t) \quad \mbox{in}\ 
\left(H^{1}(\Omega)\right)^*
 $$
with a proper lower semicontinuous convex function $\phi$ defined well and by applying monotonicity methods (Br\'ezis \cite{Brezis}) which are useful methods for unbounded domains. 
In \cite{KY2} the growth condition for $\beta$ was imposed as 
\begin{align}\label{2-growth}
\int_0^r\beta(s)\,ds \geq c|r|^2 \quad \mbox{for all}\ r \in \mathbb{R} 
\end{align}
with some constant $c>0$, and $\beta$ admits the example  
$\beta(r) = |r|^{q-1}r + r$, 
where $q >0, q\neq1$. 
In \cite{KY1} the growth condition for $\beta$ was assumed as follows:  
\begin{align}\label{m-growth}
\int_0^r\beta(s)\,ds \geq c|r|^{m} \quad \mbox{for all}\ r \in \mathbb{R} 
\end{align}
with some constant $c>0$ and $m>1$, and 
$\beta$ includes the typical example 
$$
\beta(r) = |r|^{q-1}r,
$$
where $q>0$ ($q > 1$: the porous media equation 
(see, e.g., \cite{ASS-2016, M-2010, V-2007, Y-2008}), 
$0<q<1$: the fast diffusion equation 
(see, e.g., \cite{B-1983, RV-2002, V-2007})). 
However,  
the examples in \cite{KY1, KY2} exclude the Stefan problem 
(see, e.g., \cite{BA-2005, D-1977, Fri-1968, FKP-2004, F-2015, HK-1991}): 
$$
\beta(r) = 
\begin{cases}
k_s r &\mbox{if}\ r<0, \\
0     &\mbox{if}\ 0 \leq r \leq L, \\
k_{\ell}(r-L) &\mbox{if}\ r > L, 
\end{cases} 
$$
since this $\beta$ 
does not satisfy \eqref{2-growth}, \eqref{m-growth}. 
This is due to a direct approach to \ref{P} in \cite{KY1, KY2}.

\newpage

The purpose of this paper is to 
remove the growth condition for $\beta$ such as 
\eqref{2-growth}, \eqref{m-growth} completely 
and provide a new existence result for \ref{P}.  
To this end we turn our eyes to the fact that 
\ref{Pep} is solvable without such growth condition for $\beta$ 
by the help of the approximation term 
$\ep(-\Delta+1)u_{\ep} + \pi_{\ep}(u_{\ep})$ 
and regard \ref{P} as an asymptotic limit of \ref{Pep} as $\ep \searrow 0$.  
As a consequence, 
the Stefan problem can be included in examples of \ref{P} 
even if $\Omega$ is {\it unbounded}.   
To describe the result we introduce conditions, notations and definitions. 
We will assume the following four conditions:

%
%
%
 \begin{enumerate} 
 \item[(C1)] The following conditions (C1a) and (C1b) hold: 
 \begin{enumerate}
 \item[(C1a)] $\beta : \mathbb{R} \to \mathbb{R}$                                
 is a single-valued maximal monotone function 
 and 
 $$
 \beta(r) = \hat{\beta}\,'(r) = \partial\hat{\beta}(r), 
 $$
 where 
 $\hat{\beta}\,'$ and $\partial\hat{\beta}$ 
respectively denote the differential and  
 subdifferential 
 of a proper differentiable (lower semicontinuous) convex function 
 $\hat{\beta} : \mathbb{R} \to [0, +\infty]$ 
 satisfying $\hat{\beta}(0) = 0$. 
 This entails $\beta(0) = 0$. 
 \item[(C1b)] 
 For all $z \in H^1(\Omega)$, if $\hat{\beta}(z) \in L^1(\Omega)$,  
 then $\beta(z) \in L_{{\rm loc}}^1(\Omega)$.  
 For all $z \in H^1(\Omega)$ and 
 all $\psi \in C_{\mathrm{c}}^{\infty}(\Omega)$, 
 if $\hat{\beta}(z)\in L^1(\Omega)$,  
 then $\hat{\beta}(z + \psi) \in L^1(\Omega)$.
 \end{enumerate}
 \item[(C2)] $g\in L^2\bigl(0, T; L^2(\Omega)\bigr)$. 
 Then we fix a solution 
 $f\in L^2\bigl(0, T; H^2(\Omega)\bigr)$ of 
 \begin{equation*}
     \begin{cases}
         (-\Delta+1)f(t) = g(t) 
         & \mbox{a.e.\ on}\ \Omega,
     \\[2mm] 
         \partial_{\nu}f(t) = 0 
         & \mbox{in the sense of traces on}\ \partial\Omega 
     \end{cases}
 \end{equation*}
 for a.a.\ $t\in(0, T)$, that is, 
 \begin{equation*}
 \int_{\Omega}\nabla f(t)\cdot\nabla z + 
 \int_{\Omega}f(t)z\ = \int_{\Omega}g(t)z \quad 
 \mbox{for all}\ z\in H^1(\Omega). 
 \end{equation*}
 \item[(C3)] $\pi_{\ep} : \mathbb{R} \to \mathbb{R}$ is                          
 a Lipschitz continuous function and $\pi_{\ep}(0) = 0$ 
 for all $\ep\in(0, 1]$. 
 Moreover, there exist a constant $c_{1}>0$ and 
 a strictly increasing continuous function 
 $\sigma : [0, 1] \to [0, 1]$ such that 
 $\sigma(0) = 0$, $\sigma(1) = 1$, $c_1\sigma(\ep) < \ep$ and 
   \begin{equation*}
   \bigl|\pi_{\ep}'\bigr|_{L^{\infty}(\mathbb{R})} 
   \leq c_{1}\sigma(\ep) 
   \quad \mbox{for all}\ \ep\in(0, 1].
   \end{equation*} 
 Moreover, $r\mapsto\frac{\ep}{2}r^2+\hat{\pi}_{\ep} (r)$ 
 is convex for all $\ep \in (0, 1]$, where 
 $\hat{\pi}_{\ep} (r):=\int_{0}^{r}\pi_{\ep}(s)\,ds$.
 \item[(C4)] $u_{0}\in L^2(\Omega)$ 
 and 
 $\hat{\beta}(u_{0})\in L^1(\Omega)$. 
 Also, $u_{0\ep}\in H^1(\Omega)$ 
 fulfills $\hat{\beta}(u_{0\ep})\in L^1(\Omega)$, 
 $|u_{0\ep}|_{L^2(\Omega)}^2 \leq c_{2}$, 
 $\int_{\Omega}\hat{\beta}(u_{0\ep}) \leq c_{2}$,  
 $\ep|\nabla u_{0\ep}|_{(L^2(\Omega))^N}^2 \leq c_{2}$ 
 for all $\ep\in(0, 1]$, where $c_{2} > 0$ 
 is a constant independent of $\ep$; 
 in addition, 
 $u_{0\ep}\to u_{0}$ in $L^2(\Omega)$ as 
 $\ep\searrow0$.
 \end{enumerate}
 %
 %
 %
\begin{remark}
The condition (C1b) and 
the convexity of $r\mapsto\frac{\ep}{2}r^2+\hat{\pi}_{\ep} (r)$ 
in the condition (C3) are useful 
in proving that $(-\Delta+1)\mu_\ep$ in \ref{Pep} 
can be represented by a subdifferential of some convex function 
when $\Omega$ is unbounded (see \cite[Lemma 4.2]{KY1}). 
Also, in this paper, it is an essential assumption that 
$\beta$ is single-valued. The multi-valued case will be discussed 
in our future work. 
Moreover, the condition for $\Omega$ 
is assumed in order to use the elliptic regularity.
\end{remark}

 %
 %
 %
We put the spaces $H, V, W$ as follows: 
   $$
   H:=L^2(\Omega), \quad V:=H^1(\Omega), \quad 
   W:=\bigl\{z\in H^2(\Omega)\ |\ \partial_{\nu}z = 0 \quad 
   \mbox{a.e.\ on}\ \partial\Omega\bigr\}. 
   $$
 Then $H$ and $V$ are Hilbert spaces 
 with inner products $(\cdot, \cdot)_H$ and $(\cdot, \cdot)_V$, respectively. 
 The notation $V^{*}$ denotes the dual space of $V$ with 
 duality pairing $\langle\cdot, \cdot\rangle_{V^*, V}$. 
 Moreover we define a bijective mapping $F : V \to V^{*}$ and 
 an inner product in $V^{*}$ as 
    \begin{align}
    &\langle Fv_{1}, v_{2} \rangle_{V^*, V} := 
    (v_{1}, v_{2})_{V} \quad \mbox{for all}\ v_{1}, v_{2}\in V, 
    \label{defF}
    \\[1mm]
    &(v_{1}^{*}, v_{2}^{*})_{V^{*}} := 
    \left\langle v_{1}^{*}, F^{-1}v_{2}^{*} 
    \right\rangle_{V^*, V} 
    \quad \mbox{for all}\ v_{1}^{*}, v_{2}^{*}\in V^{*};
    \label{innerVstar}
    \end{align}
 note that $F : V \to V^{*}$ is well-defined by 
 the Riesz representation theorem. 
 We remark that (C2) implies 
 $Ff(t)=g(t)$ for a.a.\ $t\in(0, T)$.

We define weak solutions of \ref{P} as follows.
%
%
%
 \begin{df}         
 A pair $(u, \mu)$ with 
    \begin{align*}
    &u\in H^1(0, T; V^{*})\cap L^{\infty}(0, T; H), 
    \\
    &\mu\in L^2(0, T; V)
    \end{align*}
 is called a {\it weak solution} of \ref{P} 
 if $(u, \mu)$ satisfies 
    \begin{align}
        & \bigl\langle u'(t), z\bigr\rangle_{V^{*}, V} + 
          \bigl(\mu(t), z\bigr)_{V} = 0 \quad 
          \mbox{for all}\ z \in V\ \mbox{and a.a}.\ t\in(0, T), 
          \label{de4}
     \\[0mm]
        & \mu(t) = \beta(u(t)) - f(t) \quad \mbox{in}\ V 
          \quad \mbox{for a.a.}\ t\in(0, T), \label{de5}
     \\[0mm]
        & u(0) = u_{0} \quad \mbox{a.e.\ on}\ \Omega. \label{de6}
     \end{align}
 \end{df}

Now the main result reads as follows: 
 \begin{thm}\label{maintheorem1}
 Let $T>0$. 
 Assume {\rm (C1)-(C4)}. 
 Then there exists a unique weak solution $(u, \mu)$ of {\rm \ref{P}}, 
 satisfying                                                 
     \begin{equation*}
        u \in H^1(0, T; V^{*})\cap L^{\infty}(0, T; H),
         \quad \mu \in L^2(0, T; V) 
     \end{equation*}
 and there exists a constant $M>0$ 
 such that 
     \begin{align}
     &|u(t)|_{H}^2 
     \leq M, \label{es0} 
     \\[1mm]
     &\int_{0}^{t}\bigl|u'(s)\bigr|_{V^{*}}^2\,ds 
     \leq M, \label{es1} 
     \\
     &\int_{0}^{t}|\mu(s)|_{V}^2\,ds \leq M, \label{es2} 
     \\
     &\int_{0}^{t}|\beta(u(s))|_{V}^2\,ds 
     \leq M \label{es3}
     \end{align}
 for all $t\in[0, T]$. 
 Moreover, in {\rm (C4)} assume further that 
 $$|u_{0\ep}-u_{0}|^2_{V^{*}} \leq c_{3}\ep^{1/2}$$ 
 for some constant $c_{3} > 0$ and 
 let $(u_{\ep}, \mu_{\ep})$ be a weak solution of {\rm \ref{Pep}} 
 for $\ep \in (0, \overline{\ep}]$ (see Section \ref{Sec3} below).
 Then there exists a constant $C^{*}>0$ such that 
 for all $\ep \in (0, \overline{\ep}]$, 
     \begin{equation}\label{eres}
        |u_{\ep}-u|^2_{C([0, T]; V^{*})} + 
        2\int_{0}^{T} 
                   \bigl(\beta(u_{\ep}(s)) - \beta(u(s)), u_{\ep}(s) - u(s)\bigr)_{H}\,ds 
        \leq C^{*}\ep^{1/2}.  
     \end{equation} 
 \end{thm}

\smallskip

%
%
%
\begin{remark}
The operator $-\Delta + 1$ in \ref{P} and \ref{Pep} corresponds to 
the Riesz isomorphism from $V$ onto $V^*$. 
In the case of bounded domains, 
``$+1$'' of the operator $-\Delta + 1$ can be removed by virtue of 
the Poincar\'e--Wirtinger inequality (see e.g., 
\cite{CF-2015, Kenmochi-1990, KL-2005}).  
However, since the domain $\Omega$ is unbounded and the function $\beta$ 
is nonlinear in this paper, 
it would be difficult to remove ``$+1$'' of the operator $-\Delta + 1$ 
in \ref{P} and \ref{Pep}, which is an open question; note that 
the methods of \cite{CF-2015, Kenmochi-1990, KL-2005} cannot be 
applied in this paper because $|\Omega|$ appears in these methods, 
for example, the projection  
$$
Pz = z - \frac{1}{|\Omega|}\int_{\Omega} z(x)\,dx, \quad z \in H
$$ 
was effectively used.
\end{remark}

%
%
%
The strategy of the proof of Theorem \ref{maintheorem1} is as follows. 
The advantage of our approach from the Cahn--Hilliard system 
as in \cite{CF-2016, F-2017} 
is to obtain estimates, independent of $\ep>0$, for solutions to \ref{Pep} 
without any growth condition for $\beta$ (Lemma \ref{solPep}).  
The main part of this paper is to confirm Cauchy's criterion 
for solutions of \ref{Pep} (Lemma \ref{Cauchy}) 
and to obtain existence and estimates for \ref{P} 
without the growth condition for $\beta$ 
by passing to the limit in the approximate problem \ref{Pep} 
as $\ep \searrow 0$. 

%
%
%

The plan of this paper is as follows.  
Section \ref{Sec2} presents 
the porous media equation, the fast diffusion equation and the Stefan problem as examples. 
Section \ref{Sec3} provides the result for \ref{Pep}.  
In Section \ref{Sec4} we verify Cauchy's criterion 
of solutions to \ref{Pep} and prove Theorem \ref{maintheorem1}.

 \section{Examples}\label{Sec2}

\begin{ex1}
In \ref{P} and \ref{Pep} we consider 
$$
\beta (r) = |r|^{q-1}r \quad (q > 0), \qquad 
\pi_{\ep}(r) = -\frac{\ep}{2} r.           
$$         
In the case that $q>1$, the above function $\beta$ 
appears in the porous media equation 
 (see e.g., \cite{ASS-2016, M-2010, V-2007, Y-2008}). 
In the case that $0<q<1$, 
$\beta$ is the function in the fast diffusion equation 
 (see e.g., \cite{B-1983, RV-2002, V-2007}). 
Also, 
$\pi_{\ep}$ is the function 
appearing in the Cahn--Hilliard equations. 
In both examples,  
$\beta$ and $\pi_{\ep}$ satisfy (C1), (C3) and for $\ep>0$ 
there exists $u_{0\ep}$ 
satisfying (C4) and the assumption of Theorem \ref{maintheorem1} 
(see \cite[Section 6]{KY1}). 
\end{ex1}

\begin{ex2}
The Stefan problem mathematically describes 
the solid-liquid phase transition. 
The problem is described by \ref{P} with 
$$
\beta(r) = 
\begin{cases}
k_s r &\mbox{if}\ r<0, \\
0     &\mbox{if}\ 0 \leq r \leq L, \\
k_{\ell}(r-L) &\mbox{if}\ r > L,
\end{cases} 
\qquad 
\pi_{\ep}(r) = -\frac{\ep}{2} r 
$$
for all $r \in \mathbb{R}$, where $k_s, k_{\ell} > 0$ stand for the heat conductivities 
on the solid and liquid regions, respectively; 
$L > 0$ is the latent heat coefficient. In this model, $u$ and $\beta(u)$ 
represent the enthalpy and the temperature, respectively 
(see e.g., \cite{BA-2005, D-1977, Fri-1968, FKP-2004, F-2015, HK-1991}). 
In this case 
we can confirm that $\beta$ and $\pi_{\ep}$ satisfy (C1) and (C4) as follows. 

It follows from a direct computation that  
$$
\beta(r) = \hat{\beta}'(r) = \partial \hat{\beta}(r), 
$$
where 
$$
\hat{\beta}(r) := 
\begin{cases}
\frac{k_s}{2} r^2 &\mbox{if}\ r<0, \\
0     &\mbox{if}\ 0 \leq r \leq L, \\
\frac{k_{\ell}}{2}(r-L)^2 &\mbox{if}\ r > L.
\end{cases}
$$
Let $z \in V=H^1(\Omega)$ and let $K \subset \Omega$ be compact. 
Then we have 
\begin{align*}
\int_K \beta(z) 
= k_s\int_{K \cap [z < 0]} z 
   + k_{\ell}\int_{K \cap [z > L]} (z-L)
\leq (k_s + k_{\ell})|K|^{1/2}|z|_{L^2(\Omega)} 
<\infty. 
\end{align*}
Thus $\beta(z) \in L^1_{\mathrm{loc}}(\Omega)$. 
Also, letting $z \in V$ and $\psi \in C^{\infty}_{\mathrm c}(\Omega)$,  
we derive that  
\begin{align*}
\int_{\Omega} \hat{\beta}(z+\psi) 
&= \frac{k_s}{2}\int_{[z+\psi < 0]} (z+\psi)^2 
   + \frac{k_{\ell}}{2}\int_{[z+\psi > L]} (z+\psi-L)^2 \\
&\leq \frac{k_s}{2}\int_{\Omega} (z+\psi)^2 
   + \frac{k_{\ell}}{2}\int_{[z+\psi > L]} (z+\psi)^2 \\
&\leq (k_s + k_{\ell})(|z|_{L^2(\Omega)}^2 
                                     + |\psi|_{L^2(\Omega)}^2) \\
&< \infty, 
\end{align*}
which implies $\hat{\beta}(z + \psi) \in L^1(\Omega)$. 
Hence (C1) holds. 

To verify (C4) we let $u_0 \in H=L^2(\Omega)$ with 
$\hat{\beta}(u_0) \in L^1(\Omega)$ and put 
 \begin{align*}
 &A:=-\Delta+I 
 : D(A):=W \subset H \to H, 
 \\[1mm]
 &J_{\ep}:=(I+\ep A)^{-1}, \quad \ep>0.
 \end{align*}
 Then there exists $u_{0\ep} \in H^2(\Omega)$ such that 
    \begin{equation*}
       \begin{cases}
         u_{0\ep} + \ep(-\Delta + 1)u_{0\ep} = u_{0} 
         \quad \mbox{in}\ \Omega,
         \\[2mm]
         \partial_{\nu}u_{0\ep} = 0 
         \quad \mbox{on}\ \partial\Omega, 
       \end{cases}
    \end{equation*}
 that is, 
 $$u_{0\ep} = J_{\ep}u_{0}.$$ 
 The properties of $J_{\ep}$ yield that  
  \begin{align*}
  &u_{0\ep} = J_{\ep}u_{0} \to u_{0} 
  \quad \mbox{in}\ H\ \mbox{as}\ \ep \searrow 0, 
  \\[1mm]
  &|u_{0\ep}|_{H} = |J_{\ep}u_{0}|_{H} 
  \leq |u_{0}|_{H}, 
  \end{align*}
  and hence 
  \begin{align}
  &\int_{\Omega}\hat{\beta}(u_{0\ep}) \notag
  \leq \frac{1}{2}\max\{k_s, k_{\ell}\}|u_{0\ep}|_{H}^{2} 
  \leq \frac{1}{2}\max\{k_s, k_{\ell}\}|u_0|_{H}^{2}, 
  \\[2mm]
  &\ep|u_{0\ep}|_{V}^2     \label{6.1}
  = \bigl(\ep(-\Delta+I)u_{0\ep}, u_{0\ep}\bigr)_{H} 
  = (u_{0}-u_{0\ep}, u_{0\ep})_{H} 
  \leq |u_{0}|_{H}^2. 
  \end{align}
 Thus there exists $u_{0\ep}$ satisfying (C4). 
 Moreover, we observe that 
 $$
 |u_{0\ep}-u_{0}|_{V^*} \leq \ep^{1/2}|u_{0}|_{H}.
 $$ 
 Indeed, it follows from \eqref{6.1} that 
 $$
 |u_{0\ep}-u_{0}|_{V^*}^2 
 = |\ep(-\Delta+I)u_{0\ep}|_{V^*}^2 
 = \ep^2|Fu_{0\ep}|_{V^*}^2 
 = \ep^2|u_{0\ep}|_{V}^2 
 \leq \ep|u_{0}|_{H}^2.
 $$
 Finally, letting $g \in L^2\bigl(0,T; L^2(\Omega)\bigr)$, 
 we can see that (C2) is satisfied. 
 Also, we can confirm (C3) in view of the definition of $\pi_{\ep}$. 

 Therefore (C1)-(C4) hold and 
 we can apply Theorem \ref{maintheorem1} 
 for the above $\beta$ and $\pi_{\ep}$.
\end{ex2}

\section{Preliminaries}\label{Sec3}

In this section we introduce the definition of weak solutions to \ref{Pep} and show 
the result for existence of weak solutions to \ref{Pep}  
with uniform estimates in $\ep$.

%
%
%
%
 \begin{df}        
 Let $T>0$. 
 A pair $(u_{\ep}, \mu_{\ep})$ with 
    \begin{align*}
    &u_{\ep}\in H^1(0, T; V^{*})\cap L^{\infty}(0, T; V)\cap 
    L^2(0, T; W), 
    \\
    &\mu_{\ep}\in L^2(0, T; V)
    \end{align*}
 is called a {\it weak solution} of \ref{Pep} if 
 $(u_{\ep} ,\mu_{\ep})$ 
 satisfies 
    \begin{align}
        & \bigl\langle u_{\ep}'(t), z\bigr\rangle_{V^{*}, V} + 
          \bigl(\mu_{\ep}(t), z\bigr)_{V} = 0 \quad 
          \mbox{for all}\ z \in V\ 
          \mbox{and a.a}.\ t\in(0, T), \label{de7}
     \\[3mm]
        & \mu_{\ep}(t) = \ep(-\Delta + I)u_{\ep}(t) + 
        \beta(u_{\ep}(t)) + \pi_{\ep}(u_{\ep}(t)) - f(t) 
        \quad \mbox{in}\ V
        \quad \mbox{for a.a.}\ t\in(0, T), \label{de8}
     \\[3mm]
        & u_{\ep}(0) = u_{0\ep} 
        \quad \mbox{a.e.\ on}\ \Omega. \label{de9}
     \end{align}
 \end{df}

%
%
 \begin{lem}\label{solPep}
 Let $T>0$. 
 Assume {\rm (C1)-(C4)}. Then for every $\ep\in(0, 1]$                         
 there exists a unique weak solution $(u_{\ep}, \mu_{\ep})$ of 
 {\rm \ref{Pep}}, satisfying
     \begin{equation*}
        u_{\ep} \in H^1(0, T; V^{*})
        \cap L^{\infty}(0, T; V)\cap L^2(0, T; W), 
        \quad \mu_{\ep} \in L^2(0, T; V) 
     \end{equation*}
 and there exist  constants $M>0$ and 
 $\overline{\ep}\in(0, 1]$ 
 such that 
     \begin{align}
     &|u_{\ep}(t)|_{H}^2 
        + \ep\int_{0}^{t}|(-\Delta+I)u_{\ep}(s)|_{H}^2\,ds 
     \leq M, \label{epes0} 
     \\
     &\int_{0}^{t}\bigl|u_{\ep}'(s)\bigr|_{V^{*}}^2\,ds + 
     \ep|u_{\ep}(t)|_{V}^2 
     \leq M, \label{epes1}
     \\
     &\int_{0}^{t}|\mu_{\ep}(s)|_{V}^2\,ds \leq M, 
     \label{epes2}
     \\
     &\int_{0}^{t}|\beta(u_{\ep}(s))|_{H}^2\,ds 
     \leq M \label{epes3}
     \end{align}
 for all $t\in[0, T]$ and all $\ep\in(0, \overline{\ep}]$.
 \end{lem}
\begin{proof}
We can prove existence and estimates for \ref{Pep} 
by setting the proper lower semicontinuous convex function 
$\phi_{\ep} : V^{*}\to\overline{\mathbb{R}}$ as 
$$
 \phi_{\ep}(z):=
   \begin{cases}
   \displaystyle\frac{\ep}{2}\int_{\Omega}
                   \bigl(|z(x)|^2 + |\nabla z(x)|^2\bigr)\,dx + 
   \int_{\Omega}\hat{\beta}(z(x))\,dx + 
   \int_{\Omega}\hat{\pi}_{\ep} (z(x))\,dx &
   \\
   \hspace{50mm}\mbox{if}\ 
   z\in D(\phi_{\ep}):=\{z\in V\ |\ \hat{\beta}(z)\in L^1(\Omega) \},\ &
   \\[3mm]
   +\infty 
   \hspace{42.7mm}\mbox{otherwise}  &
   \end{cases}
 $$
and by applying the monotonicity method for 
\begin{equation*}
       \begin{cases}
         u_{\ep}'(t) + \partial\phi_{\ep}(u_{\ep}(t)) = g(t) 
         \quad \mbox{in} \ V^{*}
         \quad \mbox{for a.a.}\ t\in[0, T],
         \\[2mm]
         u_{\ep}(0) = u_{0\ep} \quad \mbox{in}\ V^{*}. 
       \end{cases}
    \end{equation*}
Indeed, 
\cite[Lemma 4.1]{KY1} assures that $\phi_{\ep}$ is 
proper lower semicontinuous convex on $V^*$;  
note that
in \cite[Lemma 4.1]{KY1} 
the growth condition for $\beta$ does not need to obtain the inequality 
\begin{align*}
 \lambda 
 \geq \phi_{\ep}(z_{n}) 
 = \frac{\ep}{2}|z_{n}|_{V}^2 + \int_{\Omega}\hat{\beta}(z_{n}) 
    + \int_{\Omega}\hat{\pi}_{\ep} (z_{n}) 
 \geq \frac{\ep}{2}|z_{n}|_{V}^2 + \int_{\Omega}\hat{\pi}_{\ep} (z_{n}) 
 \end{align*} 
because of the nonnegativity of $\hat{\beta}$. 
Thus
we can prove this lemma in the same way as in \cite[Section 4]{KY1} 
(without growth conditions for $\beta$).
\end{proof}

\vspace{10pt}

\section{Proof of Theorem \ref{maintheorem1}}\label{Sec4}

This section gives the proof of 
Theorem \ref{maintheorem1} 
by confirming that the solution of \ref{Pep} converges to 
a function as $\ep \searrow 0$, which constructs the solution of \ref{P}. 
The key is to show 
the following lemma which asserts Cauchy's criterion for solutions of \ref{Pep}. 
 \begin{lem}\label{Cauchy}                                                                   
 Let $\overline{\ep}$, $(u_{\ep}, \mu_{\ep})$ and $M$ be as in Lemma \ref{solPep}. 
 Then we have 
 \begin{align}\label{Cauchyineq}
   &|u_{\ep}-u_{\gamma}|^2_{C([0, T]; V^{*})} + 
        2\int_{0}^{T} 
        \bigl(\beta(u_{\ep}(s)) - \beta(u_{\gamma}(s)), 
                              u_{\ep}(s) - u_{\gamma}(s)\bigr)_{H}
        \,ds \\ \notag
   &\leq |u_{0\ep} - u_{0\gamma}|_{V^*}^2 
            + 2M(\ep^{1/2} + \gamma^{1/2}) 
            + 2MT(\ep^{1/2} + \gamma^{1/2} + 2c_1(\sigma(\ep) + \sigma(\gamma)))
 \end{align}
 for all $\ep, \gamma \in (0, \overline{\ep}]$.
 \end{lem}
\begin{proof}
We have from \eqref{defF}, \eqref{innerVstar} and \eqref{de7} that 
\begin{align}\label{Ca1}
\frac{1}{2}\frac{d}{ds}|u_{\ep}(s) - u_{\gamma}(s)|_{V^*}^2 
&= \langle u_{\ep}'(s) - u_{\gamma}'(s), F^{-1}(u_{\ep}(s) - u_{\gamma}(s)) 
                                                                                        \rangle_{V^*, V}
\\ \notag
&= -(F^{-1}(u_{\ep}(s) - u_{\gamma}(s)), \mu_{\ep}(s) - \mu_{\gamma}(s))_{V} 
\\ \notag
&= -\langle u_{\ep}(s) - u_{\gamma}(s), \mu_{\ep}(s) - \mu_{\gamma}(s) 
                                                                                        \rangle_{V^*, V} 
\\ \notag
&= -(u_{\ep}(s) - u_{\gamma}(s), \mu_{\ep}(s) - \mu_{\gamma}(s))_{H}. 
\end{align}
Here \eqref{de8} yields that 
\begin{align}\label{Ca2'}
&-(u_{\ep}(s) - u_{\gamma}(s), \mu_{\ep}(s) - \mu_{\gamma}(s))_{H} 
\\ \notag
&= (u_{\ep}(s) - u_{\gamma}(s), -\ep(-\Delta + 1)u_{\ep}(s) 
                                                      + \gamma(-\Delta + 1)u_{\gamma}(s))_{H}
\\ \notag
&\quad\,-(\beta(u_{\ep}(s)) - \beta(u_{\gamma}(s)), 
                                                        u_{\ep}(s) - u_{\gamma}(s))_{H} 
\\ \notag
&\quad\,+(u_{\ep}(s) - u_{\gamma}(s), -\pi_{\ep}(u_{\ep}(s)) 
                                                            + \pi_{\gamma}(u_{\gamma}(s)))_{H}. 
\end{align}
Combination of \eqref{Ca1} and \eqref{Ca2'} 
together with the Schwarz inequality 
gives that 
\begin{align}\label{Ca2}
&\frac{1}{2}\frac{d}{ds}|u_{\ep}(s) - u_{\gamma}(s)|_{V^*}^2 
+(\beta(u_{\ep}(s)) - \beta(u_{\gamma}(s)), 
                                                        u_{\ep}(s) - u_{\gamma}(s))_{H}
\\ \notag
&= (u_{\ep}(s) - u_{\gamma}(s), -\ep(-\Delta + 1)u_{\ep}(s) 
                                                      + \gamma(-\Delta + 1)u_{\gamma}(s))_{H} 
\\ \notag
&\quad\,
  +(u_{\ep}(s) - u_{\gamma}(s), -\pi_{\ep}(u_{\ep}(s)) 
                                                            + \pi_{\gamma}(u_{\gamma}(s)))_{H}
\\ \notag
&\leq (|u_{\ep}(s)|_{H} + |u_{\gamma}(s)|_{H})
(\ep|(-\Delta+1)u_{\ep}(s)|_{H} + \gamma|(-\Delta+1)u_{\gamma}(s)|_{H}) 
\\ \notag
&\quad\,+ (|u_{\ep}(s)|_{H} + |u_{\gamma}(s)|_{H})
                     (|\pi_{\ep}(u_{\ep}(s))|_{H} + |\pi_{\ep}(u_{\ep}(s))|_{H}). 
\end{align}
Moreover, it follows from \eqref{epes0} and (C3) that 
\begin{align}\label{Ca3}
&|u_{\ep}(s)|_{H} \leq \sqrt{M}, \\ \label{Ca31}
&|\pi_{\ep}(u_{\ep}(s))|_{H} \leq c_1\sigma(\ep)|u_{\ep}(s)|_{H} 
\leq c_1\sigma(\ep)\sqrt{M} 
\end{align}
for all $s \in [0, T]$ and all $\ep \in (0, \overline{\ep}]$. 
Thus we have from \eqref{Ca2}, \eqref{Ca3} and \eqref{Ca31} that 
\begin{align*}
&\frac{1}{2}\frac{d}{ds}|u_{\ep}(s) - u_{\gamma}(s)|_{V^*}^2 
+(\beta(u_{\ep}(s)) - \beta(u_{\gamma}(s)), 
                                                        u_{\ep}(s) - u_{\gamma}(s))_{H}
\\
&\leq 2\sqrt{M}
         (\ep|(-\Delta+1)u_{\ep}(s)|_{H} + \gamma|(-\Delta+1)u_{\gamma}(s)|_{H} 
                                   + c_1\sqrt{M}\sigma(\ep) + c_1\sqrt{M}\sigma(\gamma)) 
\\ 
&\leq M(\ep^{1/2} + \gamma^{1/2}) +\ep^{3/2}|(-\Delta+1)u_{\ep}(s)|_{H}^2 
         + \gamma^{3/2}|(-\Delta+1)u_{\gamma}(s)|_{H}^2 
\\
         &\quad\,+ 2c_1M(\sigma(\ep) + \sigma(\gamma)).
\end{align*}
Hence, integrating this inequality, we conclude from \eqref{epes0} that 
\eqref{Cauchyineq} holds.
\end{proof}

We are now in a position to complete the proof of Theorem \ref{maintheorem1}. 

\begin{prth1.1a}
Lemma \ref{Cauchy}, the monotonicity of $\beta$, (C3), (C4) imply that 
$\{u_{\ep}\}_{\ep \in (0, \overline{\ep}]}$ 
satisfies Cauchy's criterion in $C([0, T]; V^*)$, 
and hence there exists a function $u \in C([0, T]; V^*)$ such that 
\begin{align}\label{SK}
u_{\ep} \to u \quad \mbox{in}\ C([0, T]; V^*)
\end{align}
as $\ep \searrow 0$.
We have from \eqref{SK} and (C4) that 
$$
u(0) = u_0 \quad \mbox{in}\ V^*
$$
and, since $u_0 \in H$, it holds that 
\begin{align}\label{shokiti}
u(0) = u_0 \quad \mbox{a.e.\ on}\ \Omega.
\end{align}
The estimates \eqref{epes0}--\eqref{epes3} yield that 
there exist a subsequence $\{\ep_{k}\}_{k\in\mathbb{N}}$, 
with $\ep_{k} \searrow 0$ as $k \to \infty$, and 
some functions $v \in H^1(0, T; V^*) \cap L^{\infty}(0, T; H)$, 
$\mu \in L^2(0, T; V)$ and $\xi \in L^2(0, T; H)$ satisfying  
\begin{align}
&u_{\ep_k} \to v \quad \mbox{weakly$^*$ in}\ H^1(0, T; V^*) 
\cap L^{\infty}(0, T; H), 
\label{shun0}\\
&\ep_{k}(-\Delta+I)u_{\ep_{k}} \to 0\quad \mbox{in}\ L^2(0, T; H), 
\label{shun2}\\
&\mu_{\ep_{k}} \to \mu \quad \mbox{weakly in}\ L^2(0, T; V), 
\label{shun3}\\
&\beta(u_{\ep_{k}}) \to \xi \quad \mbox{weakly in}\ L^2(0, T; H) \label{shun4} 
\end{align}
as $k \to \infty$. 
Now we will confirm that 
\begin{align}\label{u=v}
u = v \quad \mbox{a.e.\ on}\ \Omega\times (0, T).
\end{align}
Let $\psi \in C_{\mathrm c}^{\infty}(\Omega\times[0, T])$. 
Then we see from \eqref{SK} that 
\begin{align}\label{pote1}
\int_{0}^{T}\int_{\Omega}(u_{\ep_{k}} - v)\psi \to \int_{0}^{T}\int_{\Omega}(u- v)\psi
\end{align}
as $k \to \infty$. 
On the other hand, from \eqref{shun0} we have  
\begin{align}\label{pote2}
\int_{0}^{T}\int_{\Omega}(u_{\ep_{k}} - v)\psi 
= \int_{0}^{T} \langle u_{\ep_{k}}(t) - v(t), \psi(t) \rangle_{V^*, V}\,dt
\to 0
\end{align}
as $k \to \infty$. 
Thus it follows from \eqref{pote1} and \eqref{pote2} that 
$$
\int_{0}^{T}\int_{\Omega}(u - v)\psi 
= 0
$$
for all $\psi \in C_{\mathrm c}^{\infty}(\Omega\times[0, T])$. 
This implies that \eqref{u=v} holds.  
Consequently, 
we derive that 
$u \in H^1(0, T; V^*) \cap L^{\infty}(0, T; H)$ and 
\begin{align}\label{shun1}
u_{\ep_k} \to u \quad \mbox{weakly$^*$ in}\ H^1(0, T; V^*) 
\cap L^{\infty}(0, T; H) 
\end{align}
as $k \to \infty$. 

Next we show that 
\begin{equation}\label{pito0}
\pi_{\ep_{k}}(u_{\ep_{k}}) \to 0 \quad \mbox{in}\ L^{\infty}(0, T; H)
\end{equation}
as $k \to \infty$.
It follows from (C3) that 
$$
|\pi_{\ep_{k}}(u_{\ep_{k}})| \leq c_1\sigma(\ep_{k})|u_{\ep_{k}}| \quad 
\mbox{a.e.\ on}\ \Omega\times(0, T), 
$$
and \eqref{epes0} enables us to see that \eqref{pito0} holds. 

Moreover, we prove that 
\begin{equation}\label{betaxi}
\xi = \beta(u) \quad \mbox{a.e.\ on}\ \Omega\times(0, T).
\end{equation}
To this end it suffices to confirm that 
\begin{equation}\label{limsup}
\limsup_{k\to\infty}\int_0^T (\beta(u_{\ep_{k}}(t)), u_{\ep_{k}}(t))_{H}\,dt 
\leq \int_0^T (\xi(t), u(t))_{H}\,dt
\end{equation}
(see \cite[Proposition 2.2, p.\ 38]{Barbu}). 
We infer from \eqref{de8}, \eqref{shun3}, \eqref{shun1} and \eqref{pito0} that 
\begin{align}\label{limsu}
&\int_0^T (\beta(u_{\ep_{k}}(t)), u_{\ep_{k}}(t))_{H}\,dt \\ \notag
&= \int_0^T (\mu_{\ep_k}(t) + f(t), u_{\ep_k}(t))_{H}\,dt 
   -\ep_k\int_0^T |u_{\ep_k}(t)|^2_{V}\,dt \\ \notag 
  &\quad\,-\int_0^T (\pi_{\ep_k}(u_{\ep_k}(t)), u_{\ep_k}(t))_{H}\,dt \\ \notag
&\leq \int_0^T \langle u_{\ep_k}(t), \mu_{\ep_k}(t) + f(t) \rangle_{V^*, V}\,dt 
         -\int_0^T (\pi_{\ep_k}(u_{\ep_k}(t)), u_{\ep_k}(t))_{H}\,dt \\ \notag
&\to \int_0^T \langle u(t), \mu(t) + f(t) \rangle_{V^*, V}\,dt 
  = \int_0^T (u(t), \mu(t) + f(t))_{H}\,dt 
\end{align}
as $k\to\infty$. 
Here, from \eqref{de8}, \eqref{shun2}, \eqref{shun4} and \eqref{pito0} we have 
\begin{equation}\label{muxif}
\mu = \xi - f \quad \mbox{a.e.\ on}\ \Omega\times(0, T). 
\end{equation} 
Thus combination of \eqref{limsu} and \eqref{muxif} leads to \eqref{limsup}, i.e., 
\eqref{betaxi}. 

Next we confirm that there exists a constant $C_1>0$ such that 
\begin{equation}\label{estibeta}
\int_0^t |\beta(u(s))|^2_{V}\,ds \leq C_1
\end{equation}
for all $t\in[0, T]$. 
We derive from \eqref{epes2} and \eqref{shun3} that 
there exists a constant $C_2 > 0$ such that 
\begin{equation}\label{apmu}
\int_0^t |\mu(s)|^2_{V}\,ds \leq C_2
\end{equation}
for all $t \in [0, T]$.
Hence, noting that $\mu \in L^2(0, T; V)$ and $f \in L^2(0, T; V)$, 
we see from \eqref{betaxi}, \eqref{muxif} and \eqref{apmu} 
that $\beta(u) \in L^2(0, T; V)$ and 
\begin{align*}
\int_0^t |\beta(u(s))|^2_{V}\,ds 
&= \int_0^t |\mu(s) + f(s)|^2_{V}\,ds \\
&\leq 2\int_0^t |\mu(s)|^2_{V}\,ds + 2\int_0^t |f(s)|^2_{V}\,ds  \\
&\leq 2C_2 + 2\|f\|^2_{L^2(0, T; V)},
\end{align*}
which implies \eqref{estibeta}. 
Thus, from \eqref{de7}, \eqref{shokiti}, \eqref{shun3}, \eqref{shun1},  
\eqref{betaxi} and \eqref{muxif} 
we have shown that $(u, \mu)$ is a solution of \ref{P}. 
Moreover, by \eqref{epes0}, \eqref{epes1}, \eqref{shun1}, \eqref{estibeta} and 
\eqref{apmu} we obtain \eqref{es0}--\eqref{es3}. 

Finally, we check that the solution $(u, \mu)$ of the problem \ref{P} is unique. 
Assume that 
$(u_1, \mu_1)$ and $(u_2, \mu_2)$ are the solutions of \ref{P} 
with the same initial data. 
Then 
it follows from \eqref{defF}--\eqref{de5} that 
\begin{equation*}
(u_1'(t)-u_2'(t), Fz)_{V^*} 
+ \langle Fz, \beta(u_1(t)) - \beta(u_2(t)) \rangle_{V^*, V} = 0 
\end{equation*}
for all $z \in V$. 
Choosing $z = F^{-1}(u_1(t)-u_2(t)) \in V$, we derive from \eqref{de6} that 
$$
\frac{1}{2}|u_1(t)-u_2(t)|^2_{V^*} 
+ \int_0^t (\beta(u_1(s))-\beta(u_2(s)), u_1(s)-u_2(s))_{H}\,ds = 0
$$
for all $t \in [0, T]$. Then the second term on the left-hand side is nonnegative by virtue of 
the monotonicity of $\beta$, so that 
$$
|u_1(t)-u_2(t)|^2_{V^*} \le 0
$$ 
for all $t \in [0, T]$. 
Hence it holds that $u_1=u_2$. 
Furthermore, we infer from \eqref{de5} that 
$$
\mu_1(t) = \beta(u_1(t)) -f(t) = \beta(u_2(t)) -f(t) = \mu_2(t)
$$
for all $t \in [0, T]$.
\qed 
\end{prth1.1a} 
\begin{prth1.1b}
The error estimate in \eqref{eres} 
can be proved by
the same argument as in the proof of \cite[Theorem 1.3]{KY1}.  
\qed 
\end{prth1.1b}
 \vspace{10pt}


%
 {\small
}%
\end{document}